\documentclass[12pt,a4paper]{amsart}

\usepackage{amsfonts, amsmath, amssymb, amsthm, amscd, hyperref}

\usepackage{a4wide}

\newtheorem{thm}{Theorem}
\newtheorem{lem}{Lemma}

\newtheorem{cor}[thm]{Corollary}

\theoremstyle{definition}
\newtheorem{defn}{Definition}

\theoremstyle{remark}

\newcommand{\pt}{\mathrm{pt}}

\DeclareMathOperator{\uhind}{\overline{ind}}

\newcommand{\Gcat}{\mathop{G\textrm{-}\mathrm{cat}}}

\newcommand{\Agen}{\mathop{\mathcal A\textrm{-}\mathrm{gen}}}
\newcommand{\OGXgen}{\mathop{\mathcal O_G(X)\textrm{-}\mathrm{gen}}}
\newcommand{\OHXgen}{\mathop{\mathcal O_H(X)\textrm{-}\mathrm{gen}}}
\newcommand{\OSnXgen}{\mathop{\mathcal O_{\mathfrak S_n}(X)\textrm{-}\mathrm{gen}}}
\newcommand{\Ggen}{\mathop{G\textrm{-}\mathrm{gen}}}
\newcommand{\Zpkgen}{\mathop{(Z_p)^k\textrm{-}\mathrm{gen}}}

\DeclareMathOperator{\Hom}{Hom}
\DeclareMathOperator{\Map}{Map}
 
\DeclareMathOperator{\Ind}{Ind}

\renewcommand{\int}{\mathop{\rm int}}

\renewcommand{\epsilon}{\varepsilon}

\begin{document}

\title{The genus and the category of configuration spaces}
\author{R.N.~Karasev}
\email{r\_n\_karasev@mail.ru}
\address{
Roman Karasev, Dept. of Mathematics, Moscow Institute of Physics
and Technology, Institutskiy per. 9, Dolgoprudny, Russia 141700}
\thanks{This research was supported by the President of the Russian Federation grant No. MK-1005.2008.1, and partially supported by the Dynasty Foundation.}



\begin{abstract}
In this paper configuration spaces of smooth manifolds are considered. The accent is made on actions of certain groups (mostly $p$-tori) on this spaces by permuting their points. For such spaces the cohomological index, the genus in the sense of Krasnosel'skii-Schwarz, and the equivariant Lyusternik-Schnirelmann category are estimated from below, and some corollaries for functions on configuration spaces are deduced.
\end{abstract}

\maketitle

\section{Introduction}

The main subject of this paper is configuration spaces of smooth manifolds. Let $M$ be a smooth manifold. A particular case of configuration space is the space of all sequences $(x_1, \ldots, x_n)\in M^n$ of pairwise distinct points. This classical configuration space is of particular interest as the space of configurations of $n$ points in classical mechanics. Some topological properties of this space are studied in the book~\cite{fh2001}.

In this paper we are interested in the topology of the obvious action of the symmetric group $\mathfrak S_n$ on the space of $n$-configurations; we also consider the restriction of this action to some subgroups $G\subseteq \mathfrak S_n$. The space of $n$-configurations in $\mathbb R^k$ was studied in \cite{fuks1970,vass1988,rot2008}. The main object of these studies was some measure of complexity of the configuration space, such as the homological index, or the genus in the sense of Krasnosel'skii-Schwarz~\cite{kr1952,schw1957,schw1966}, see Section~\ref{genus-and-cat}. The estimates on the genus of the configuration space of $n$ pairwise distinct points in $\mathbb R^2$ were used to study the ``topological complexity'' of finding complex roots of a polynomial in~\cite{vass1988}, see also~\cite{sma1987}.

Along with the configuration space of pairwise distinct points, we consider more general configuration spaces. By a configuration space for $M$ we mean some subspace of its Cartesian power $M^n$, defined by removing configurations with some multiple point coincidences. For example, some pairs in the configuration may be required to be distinct points, or every $k$ of points in the configuration may be required not to be the same point. One example of such a configuration space was used in the author's paper~\cite{kar2008bil} and the previous studies~\cite{fartab99,far00} of billiards in smooth convex bodies in $\mathbb R^d$. In the cited papers, some estimates on the homological index (see Sections~\ref{eq-cohomology} and \ref{eq-cohomology-p-tori} for definitions) of the configuration space, combined with the Lyusternik-Schnirelmann theory, gave some lower bounds on the number of distinct closed billiard trajectories. Another application of configuration spaces with general coincidence constraints arises in the study of coincidences of maps, see some examples in~\cite{vol1992,vol2005}.

As it was already mentioned, in addition to considering the action of $\mathfrak S_n$ on configuration spaces we also consider the actions of $G\subset \mathfrak S_n$. Let the group $G$ act on the points freely and transitively. In this case the configuration space may be considered as the space of all maps $G\to M$, denote it $\Map(G, M)$. The group $G$ acts on itself by left multiplications, so $G$ acts on $\Map(G, M)$ by right multiplications. We denote this action by
$$
(g, \phi)\in G\times \Map(G, M)\mapsto \phi^g,
$$ 
by definition put $\phi^g(x) = \phi(gx)$ for any $g,x\in G$.

The paper is organized in the following way. The main tool is considering the equivariant cohomology of configuration spaces. In Section~\ref{eq-cohomology} we collect the general notions on the equivariant cohomology, define a measure of homological complexity (index) of $G$-spaces. In Section~\ref{eq-cohomology-p-tori} we focus on the case of $G$ being a $p$-torus, introduce another index and prove a new result: Theorem~\ref{indind} on estimating the index from below.

In Section~\ref{index-conf} some lower bounds of topological complexity (indexes) of particular configuration spaces with respect to the group action are given.

In Sections~\ref{symm-config} and \ref{genus-and-cat} the main results on configuration spaces are formulated and proved. The results of Section~\ref{symm-config} show the existence of a configuration with certain symmetrical equalities, the author is particularly interested in metric symmetries, but the results are stated for general functions instead of a metric. Results of Section~\ref{genus-and-cat} estimate the genus in the sense of Krasnosel'skii-Schwarz of configuration spaces and may be used to find lower bounds on the number of critical points of a smooth symmetric function using the Lyusternik-Schnirelmann theory.

Note that in this paper Theorems~\ref{constr-conf-gen} and~\ref{pwconstr-gen} and their corresponding Lemmas~\ref{constr-conf-ind} and \ref{pwconstr-ind} (see below) are actually proved for $\mathbb R^d$. It seems plausible that for a closed manifold $M$ the estimate on the index may be larger, as it was for the ``billiard'' configuration space of the sphere in~\cite{kar2008bil}, where the bound was larger by $1$ compared to Theorem~\ref{constr-conf-gen}.

The author thanks A.Yu.~Volovikov for useful discussions and remarks.

\section{Equivariant cohomology of $G$-spaces}
\label{eq-cohomology}

In this section we state some facts on the equivariant cohomology and define a homological measure of complexity of a $G$-space.

We consider topological spaces with continuous action of a finite group $G$ and continuous maps between such spaces that commute with the action of $G$. We call them $G$-spaces and $G$-maps.

The facts in this section are quite well known, see books~\cite{hsiang1975,bart1993}.

We consider the equivariant cohomology (in the sense of Borel) of $G$-spaces, defined as 
$$
H_G^*(X, M) = H^*((X\times EG)/G, M),
$$
where some $\mathbb Z[G]$-module $M$ (acted on by the fundamental group of $(X\times EG)/G$) gives the coefficients for the cohomology.

Consider the $G$-equivariant cohomology of the point $H_G^*(M) = H_G^*(\pt, M) = H^*(BG, M)$. For any $G$-space $X$ the natural map $X\to\pt$ induces the natural map of cohomology $\pi_X^* : H_G^*(M)\to H_G^*(X, M)$.

\begin{defn}
The \emph{upper cohomological index} of a $G$-space $X$ with coefficients in $M$ is the maximal $n$ such that the natural map
$$
H_G^n(M)\to H_G^n(X, M)
$$
is nontrivial. Denote the upper index $\uhind_M X = n$. Denote the supremum over all $\mathbb Z[G]$-modules
$$
\uhind_G X = \sup_M \uhind_M X.
$$
\end{defn}

If a $G$-space $X$ has fixed points, its cohomology $H^*_G(X, M)$ contains $H^*_G(M)$ as a summand, thus its upper index is obviously $+\infty$.

The following property is obvious by definition.

\begin{lem}[Monotonicity of index]
If there exists a $G$-map $f:X\to Y$ then $\uhind_M X\le \uhind_M Y$ for any coefficients $M$.
\end{lem}

The $G$-equivariant cohomology is often calculated from the spectral sequence of the fibration $X_G = (X\times EG)/G\to BG$ with fiber $X$. Here we state the lemma from~\cite[Section~11.4]{mcc2001}.

\begin{lem}
\label{specseqeq}
Let $R$ be a ring with trivial $G$-action. There exists a spectral sequence with $E_2$-term
$$
E_2^{x, y} = H^x(BG, \mathcal H^y(X, R)),
$$
that converges to the graded module, associated with the filtration of $H_G^*(X, R)$.

The system of coefficients $\mathcal H^y(X, R)$ is obtained from the cohomology $H^y(X, R)$ by the action of $G = \pi_1(BG)$. The differentials of this spectral sequence are homomorphisms of $H^*(BG, R)$-modules.
\end{lem}

In this lemma we denote the grading in the spectral sequence by $(x, y)$ (not usual), because the letter $p$ is reserved for the prime number throughout this paper.

\section{Equivariant cohomology of $G$-spaces for $G=(Z_p)^k$}
\label{eq-cohomology-p-tori}

In this section we study in greater detail the equivariant cohomology of $G$-spaces in the case $G=(Z_p)^k$.

In this section the cohomology is taken with coefficients $Z_p$, in notations we omit the coefficients. For a group $G=(Z_p)^k$ the algebra $A_G=H_G^*(Z_p)$ has the following structure (see~\cite{hsiang1975}). In the case $p>2$ it has $2k$ multiplicative generators $v_i,u_i$ with dimensions $\dim v_i = 1$ and $\dim u_i = 2$ and relations
$$
v_i^2 = 0,\quad\beta{v_i} = u_i.
$$
We denote $\beta(x)$ the Bockstein homomorphism. 

In the case $p=2$ the algebra $A_G$ is the algebra of polynomials of $k$ one-dimensional generators $v_i$.

Consider again the spectral sequence from Lemma~\ref{specseqeq}. For every term $E_r(X)$ of this spectral sequence there is a natural map $\pi^*_r : A_G\to E_r(X)$ (it maps $A_G$ to the bottom row of $E_r(X)$).

\begin{defn}
Denote the kernel of the map $\pi^*_r$ by $\Ind^r_G X$. 
\end{defn}

Let us list the properties of $\Ind^r_G X$, that are obvious by the definition. We omit the subscript $G$ when it is clear what group is meant.

\begin{itemize}
\item 
(Monotonicity) If there is a $G$-map $f:X\to Y$, then $\Ind^r X\supseteq \Ind^r Y$.

\item
$\Ind^{r+1} X$ may differ from $\Ind^r X$ only in dimensions $\ge r$.

\item
$\bigcup_r \Ind^r X = \Ind X = \ker \pi_X^* : A_G\to H_G^*(X)$.
\end{itemize}

The first property in this list is very useful to prove nonexistence of $G$-maps. Following~\cite{vol2000,vol2005} we define a numeric invariant of this system $\Ind^r X$, that is enough for us.

\begin{defn}
Put 
$$
i_G(X) = \max \{r : \Ind_G^r X = 0\}.
$$
\end{defn}

It is easy to see that $i(X)\ge 1$ for any $G$-space $X$, $i(X)\ge 2$ for a connected $G$-space $X$, and $i(X)$ may be equal to $+\infty$. The following properties are quite clear from the definition.

\begin{itemize}
\item
(Monotonicity) If there is a $G$-map $f:X\to Y$, then $i_G(X) \le i_G(Y)$.

\item
If $H^m(X) = 0$ for $m > n$, then either $i_G(X)=+\infty$ or $i_G(X)\le n+1$.

\item 
If $\tilde H^m(X) = 0$ for $m <n$, then $i_G(X)\ge n + 1$.
\end{itemize}

The following lemma from~\cite{vol2005} (Lemma~2.1) tells more about the monotonicity.

\begin{lem}
\label{same-i-map}
Let $X, Y$ be connected paracompact $G$-spaces and let $f:X\to Y$ be a $G$-map. Suppose that $i_G(X)=i_G(Y)=n+1$. Then the map $f^*: H^n(Y)\to H^n(X)$ is nontrivial.
\end{lem}

We need the following result to estimate $i_G(X)$ from below. Some special case of it was used in the proof of Theorem~4 in the paper~\cite{kar2008bil}.

\begin{thm}
\label{indind}
Let $G=(Z_p)^k$, let $G$-space $X$ be connected. Suppose the groups $H^m(X, Z_p)$ for $m < n$ are composed of finite-dimensional $Z_p[G]$-modules, induced from proper subgroups $H\subset G$. Then
$$
i_G(X)\ge n + 1.
$$ 
\end{thm}

A $Z_p[G]$-module $M$ is induced from $Z_p[H]$-module $N$ iff $M=Z_p[G]\otimes_{Z_p[H]} N$. A finite-dimensional module is induced iff it is coinduced: $M=\Hom_{Z_p[H]}(Z_p[G], N)$ (see~\cite{bro1982}, Ch.~III, Proposition~5.9).

\begin{proof}
Put $S_G=A_G$ for $p=2$, and $S_G=Z_p[u_1,\ldots,u_k]$ for $p>2$. This is a subalgebra of polynomials in $A_G$.

If a $G$-module $M$ is coinduced from an $H$-module $N$ then $H_G^*(M) = H_H^*(N)$ (see~\cite{bro1982}, Ch.~III, Proposition~6.2). In this case the Hilbert polynomial of $H_G^*(M)$ has degree no more than $k-1$. If a $G$-module is composed of induced modules, then it follows from the cohomology exact sequence that the Hilbert polynomial of its cohomology has degree $\le k-1$, or otherwise the cohomology exact sequence could not be exact for large cohomology dimensions.

Therefore, the rows $1,\ldots, n-1$ in $E_2$-term of the spectral sequence of Lemma~\ref{specseqeq} have Hilbert polynomials of degree $\le k-1$. In all next terms $E_r$ ($r\ge 2$) this fact remains true, since the dimensions $\dim E_r^{x, y}$ decrease, and the Hilbert polynomial of a row cannot get a larger degree. 

Now let us show that the image of the differentials $d_2,\ldots,d_n$ in the bottom row of the spectral sequence is zero (that is equivalent to $i(X)\ge n+1$). Indeed, if an $A_G$-module $L$ has Hilbert polynomial of degree $\le k-1$ then
$$
\Hom_{A_G}(L, A_G) = 0,
$$
since every $l\in L$ is annihilated by some nontrivial element of $S_G$ (or the Hilbert polynomial would have degree $\ge k$), and none of the nonzero elements of $A_G$ is annihilated by nontrivial elements of $S_G$.
\end{proof}

\section{Definitions of configuration spaces}

In this section we give the definitions of different configuration spaces and introduce some notation.

\begin{defn}
Let $\Delta: M\to \Map(G, M)$ be the diagonal map, i.e. $\Delta(x)$ is the constant map of $G$ to $x\in M$.
Denote $\Map_\Delta(G, M) = \Map(G, M)\setminus\Delta(M)$. This is the space of nonconstant maps $G\to M$.
\end{defn}

\begin{defn}
Denote $[n]=\{1,2,\ldots, n\}$, and denote the configuration space
$$
V(n, M) = \Map([n], M).
$$
\end{defn}

\begin{defn}
Let $\mathcal S$ be a nonempty family of subsets in $[n]$. Denote $V\left(n, M, \mathcal S\right)$ the set of maps $f:[n]\to M$ such that every $f|_S$ for $S\in\mathcal S$ is nonconstant. We call the family $\mathcal S$ a \emph{constraint system}. Denote 
$$
w(\mathcal S) = \min_{S\in\mathcal S} |S|.
$$
We consider only the nontrivial case $w(\mathcal S)\ge 2$.
\end{defn}

If the configurations are considered as maps $G\to M$, the following definition is needed.

\begin{defn}
Let $\mathcal S$ be a nonempty family of subsets in $G$, and let $\mathcal S$ be invariant with respect to $G$-action on itself. Denote $V\left(G, M, \mathcal S\right)$ the set of maps $f:G\to M$ such that every $f|_S$ for $S\in\mathcal S$ is nonconstant. Again, we call the family $\mathcal S$ a \emph{constraint system}. The number $w(\mathcal S)$ is defined as in the previous definition and is considered to be at least $2$.
\end{defn}

The constraint system $\mathcal S=\{G\}$ gives the space $\Map_\Delta(G, M)$, defined above. 

For a finite set $X$ denote $\binom{X}{w}$ the family of all $w$-element subsets of $X$. The classical configuration space, the set of $n$-tuples of pairwise distinct points in $M^n$, is denoted $V\left(n, M, \binom{[n]}{2}\right)$ in our notation. The spaces $V\left(n, M, \binom{[n]}{w}\right)$ are called configuration-like spaces in~\cite{colusk1976,vol2007}, here we call them simply configuration spaces.

\section{The index of configuration spaces}
\label{index-conf}

Now we are ready to state and prove the core results of this paper, the lower bounds on indexes of configuration spaces.

\begin{lem}
\label{conf-space-ind} Let $G=(Z_p)^k$, let $M$ be a smooth oriented (if $p\not=2$) closed manifold of dimension $d$. Then
$$
i_G\left(\Map_\Delta(G, M)\right) \ge d(p^k-1) + 1.
$$
\end{lem}

Note that in~\cite{vol1992,vol2005} the index of $\Map_\Delta(G, M)$ was estimated from above. It was needed to show that (under some additional assumptions) for any continuous map $f:X\to M$ of a $G$-space $X$ to $M$ some $G$-orbit in $X$ is mapped to a point.

\begin{proof}
Denote $q=p^k$. Obviously, $M$ contains some copy of $\mathbb R^d$ and $\Map_\Delta(G, M)$ contains the respective copy of $S=\Map_\Delta(G, \mathbb R^d)$, the latter space is homotopy equivalent to a $d(q-1)-1$-dimensional sphere. For a sphere, it is obvious to see that $i(S) = d(q-1)$, thus $i(\Map_\Delta(G, M)) \ge d(q-1)$.

Denote for brevity $X = \Map(G, M)$ and $X_\Delta = \Map_\Delta(G, M)$.

Consider the contrary: $i(X_\Delta) = d(q-1)$. In this case by Lemma~\ref{same-i-map} the inclusion map $j : S\to X_\Delta$ of the sphere induces a nontrivial map $j^*: H^{d(q-1)-1}(X_\Delta)\to H^{d(q-1)-1}(S)$. Take the fundamental class $y\in H^{d(q-1)-1}(S)$ and consider $x\in H^{d(q-1)-1}(X_\Delta)$ such that $j^*(x) = y$. From the following diagram
\begin{equation}
\begin{CD}
\label{leseq}
H^{d(q-1)}(X) @<{\pi^*}<< H^{d(q-1)}(X, X_\Delta) @<{\delta}<< H^{d(q-1)-1}(X_\Delta) @<{\iota^*}<< H^{d(q-1)-1} (X)\\
@V{j^*}VV @V{j^*}VV @V{j^*}VV @V{j^*}VV\\
H^{d(q-1)}(\mathbb R^{dq}) @<{\pi^*}<< H^{d(q-1)}(\mathbb R^{dq}, S) @<{\delta}<< H^{d(q-1)-1}(S) @<{\iota^*}<< H^{d(q-1)-1} (\mathbb R^{dq})
\end{CD}
\end{equation}
we see that $j^*(\delta(x)) = \delta(y)\not=0$.

Consider a tubular neighborhood $N(\Delta)$ of the diagonal $\Delta(M)$ in $X$. The pair $(X, X_\Delta)$ has the same cohomology as $(B(V), S(V))$, where $V$ is the normal vector bundle of the diagonal, $B(V)$ and $S(V)$ being its unit ball and unit sphere spaces. Since the tangent bundle to $\Delta(M)$ is the same as the tangent bundle of $M$, then $V$ fits to the following exact sequence
$$
0\to T(M) \to \oplus_{g\in G} T(M) \to V\to 0
$$
over $\Delta(M)$, since the restriction of the tangent bundle of $\Map(G, M)$ to $\Delta(M)$ equals $\oplus_{g\in G} TM$. Hence, the group $G$ acts naturally on $\oplus_{g\in G} TM$, and on $V$. 

The manifold $M$ is $Z_p$-oriented and the bundle $V$ is $Z_p$-oriented. Then by Thom's isomorphism $H^*(X, X_\Delta)=u H^*(M)$, where $u$ is the $d(q-1)$-dimensional fundamental class of $V$. From $\delta(x)\not=0$ it follows that $\delta(x) = au$ for some $a\in Z_p^*$. The horizontal exactness of diagram~\ref{leseq} shows that $\pi^*(au) = 0$ and $\pi^*(u)=0$.

Note that the map $\pi^*: H^{d(q-1)} (X, X_\Delta) \to H^{d(q-1)} (X)$ has to be injective, this is a consequence of the Poincar\'e duality and injectivity of the natural map $H_d(N(\Delta)) = H_d(M) \to H_d(X)$. Thus $\pi^*(u)\not=0$, that is a contradiction.
\end{proof}

\begin{lem}
\label{constr-conf-ind}
For any constraint system $\mathcal S$ on $G=(Z_p)^k$ and $d$-dimensional smooth manifold $M$
$$
i_G\left(V\left(G, M, \mathcal S\right)\right) \ge (d-1)(p-1)p^{k-1} + w(\mathcal S) - 1.
$$
\end{lem}

Note that in~\cite{colusk1976,vol2005,vol2007} the case $M=\mathbb R^m$ and $\mathcal S = \binom{G}{w}$ was considered and some upper bounds on the index were found. The result of paper~\cite{vol2007} states that 
$$
i_G\left(V\left(G, \mathbb R^m, \binom{G}{w}\right)\right) \le (d-1)(p^k-1) + w - 1,
$$
so there is some gap between lower and upper bounds for $k>1$.

\begin{proof}
The index is monotonic, so it suffices to consider the case $M=\mathbb R^d$, which is assumed in this proof.

We are going to apply Theorem~\ref{indind}, so we have to know the cohomology $H^*\left(V(G, M, \mathcal S), Z_p\right)$ first. The coefficients $Z_p$ of the cohomology are omitted in the notation in this proof. The space $V(G, M, \mathcal S)$ is a complement to the set of linear (as subspaces of $\mathbb R^{dp^k}$) subspaces $L_S\subset \Map(G, M)$, here we denote for any $S\in\mathcal S$
$$
L_S = \{f:G\to M\ \text{such that}\  f|_S\ \text{is constant}\}.
$$
For nonempty $\mathcal T\subseteq\mathcal S$ put $L_{\mathcal T} = \bigcap_{S\in\mathcal T} L_S$.

By the result from book~\cite{gormac1988}, cited here by the review~\cite[Corollary~2]{vass2001}, the reduced cohomology of $V(G, M, \mathcal S)$ can be represented as follows
\begin{equation}
\label{gor-mac}
\tilde H^i\left(V(G, M, \mathcal S)\right) = \bigoplus_{\mathcal T\subseteq \mathcal S} H_{qd-i-\dim L_{\mathcal T} - 1} (\Delta(\mathcal T), \partial\Delta(\mathcal T)),
\end{equation}
the sum being taken over distinct subspaces $L_{\mathcal T}$. In~\cite{vass2001} this formula is proved so that the isomorphism may be not natural, so we have to be careful to introduce $G$-action on this formula. Actually this formula becomes natural, if we note that this formula is a particular case of the Leray spectral sequence for the direct image of a sheaf under the inclusion $V(G, M, \mathcal S)\to \Map(G, M)$. Hence, the reduced cohomology $\tilde H^*\left(V(G, M, \mathcal S)\right)$ should be replaced by its associated graded module, obtained from some filtering of the cohomology.

Thus the action of $G$ on the right side of Equation~\ref{gor-mac} describes the action of $G$ on the associated graded module of the cohomology. Let us study this action.

The cohomology to the right in Equation~\ref{gor-mac} is the cohomology of the order complex for the poset of spaces $L_{\mathcal U}\supseteq L_{\mathcal T}$, relative to its subcomplex, spanned by proper inclusions $L_{\mathcal U}\supset L_{\mathcal T}$. The \emph{order complex} of a poset $P$ is a simplicial complex, that has $P$ as the vertex set, and the set of chains in $P$ as the set of simplices.

In this proof, it suffices to note that the dimension of the order complexes in question is no more than $q - \dim L_{\mathcal T}/d - w(\mathcal S) + 1$.

Now consider the $G$-action on the right part of Equation~\ref{gor-mac}. If the space $L_{\mathcal T}$ is not fixed under $G$-action, then the summand, that corresponds to $L_{\mathcal T}$, has $G$-action, induced from the stabilizer of $L_{\mathcal T}$. Theorem~\ref{indind} allows us to ignore such summands. The subspaces $L_{\mathcal T}$ that are fixed under $G$-action have dimension no more than $dp^{k-1}$. Thus, they contribute to they cohomology $\tilde H^i\left(V(G, M, \mathcal S)\right)$ with the following inequality on the dimension
$$
qd - i - \dim L_{\mathcal T} - 1 \le q - \dim L_{\mathcal T}/d - w(\mathcal S) + 1,
$$
then by simple transformations 
\begin{multline*}
i\ge q(d-1) - (d-1)\dim L_{\mathcal T}/d + w(\mathcal S) - 2 \ge\\
\ge (d-1)(q-p^{k-1}) + w(\mathcal S) - 2 = (d-1)(p-1)p^{k-1} + w(\mathcal S) - 2. 
\end{multline*}

We have proved that in dimensions $i\le (d-1)(p-1)p^{k-1}+w(\mathcal S) - 1$, the right part of Equation~\ref{gor-mac} is composed of induced $Z_p[G]$-modules, but this right part is itself a decomposition of the cohomology $\tilde H^i\left(V(G, M, \mathcal S)\right)$. Thus Theorem~\ref{indind} can be applied to complete the proof.
\end{proof}

Now we are going to study the classical configuration space. Denote by $\pm Z_p$ the $\mathbb Z[\mathfrak S_n]$-module, which is $Z_p$ with the action of $\mathfrak S_n$ by the sign of permutation. In the case $p=2$, put $\pm Z_2=Z_2$.

\begin{lem}
\label{pwconstr-ind}
Let $n=p^k$ be a prime power, let $M$ be a smooth manifold of dimension $d\ge 2$, let the constraint system $\mathcal S\subset 2^n$ consist of all pairs $\mathcal S = \binom{[n]}{2}$. Then under the natural action of $\mathfrak S_n$ on the configuration space we have
$$
\uhind_N V(n, M, \mathcal S) = (d-1)(n-1),
$$
where $N=\pm Z_p$ for even $d$, and $N=Z_p$ for odd $d$.
\end{lem}

\begin{proof} Actually the proof for $d=2$ is given in~\cite{vass1988} and works for arbitrary $d\ge 2$ in the same way. Below we give a short sketch.

We go to the case $M=\mathbb R^d$, consider the configuration space $X=V(n, M, \mathcal S)$, and its one-point compactification $X'=X\cup\pt$. The Poincar\'e-Lefschetz duality tells that there is a natural isomorphism (since $N=(\pm Z_p)^{\otimes d-1}$ and $(\pm Z_p)^{\otimes d}$ gives the orientation of $X$)
$$
H_{dn-k}(X'/\mathfrak S_n, \pt, \pm Z_p) = H^k_{\mathfrak S_n} (X, N).
$$

In~\cite{fuks1970} a certain relative cellular decomposition of the pair $(X', \pt)$ that respects $\mathfrak S_n$-action was constructed. It can be described as follows: consider the graded trees, such that the grades of the vertices are $0,1,\ldots, d$, every vertex of grade $<d$ has children, all the leafs (vertices with no children) have grade $d$, and the number of leafs is $n$. We consider such a tree $T$ along with the following data: for every vertex $v\in T$ its children are ordered, and the leafs have labels $1,2,\ldots, n$ in some order. 

Say that the number $i\in[n]$ \emph{belongs to vertex} $v$, if the leaf with mark $i$ is a descendant of $v$. We say that indexes $i$ and $j$ \emph{split on level $k$}, if they belong to the same vertex $v$ of grade $k-1$, but belong to different vertices $v_i$ and $v_j$ of grade $k$. If the vertices $v_i$ and $v_j$ are ordered as $v_i<v_j$ as children of $v$, we write $i<_k j$ to show that $i$ and $j$ split on level $k$ in this given order.

Now define the open cell $C_T$ in $X$, corresponding to $T$, by the following rule: 
$$
C_T  = \{(x_1, \ldots, x_n)\in M^n : \text{if}\ i<_kj,\ \text{then}\ x_{ik} < x_{jk},\ \text{and}\ \forall l < k\ x_{il}=x_{jl} \},
$$
where $x_{ik}$ is the $k$-th coordinate of $i$-th point.

The cellular decomposition of $X$ can be described informally as follows: every two points $x_i$ and $x_j$ in a configuration $(x_1,\ldots, x_n)$ must be distinct, consider the minimal $k$ such that their coordinates $x_{ik}$ and $x_{jk}$ are different, and if $x_{ik} < x_{jk}$, say that $i<_k j$. The pattern of such relations $i<_k j$ is exactly a tree with $d+1$ levels and $n$ marked leaves at the bottom level. 

The cellular decomposition of $(X', \pt)$ is obtained by taking the closures of $C_T$. It is clear from the definition, that the dimension of a cell $C_T$ is equal to the number of vertices in $T$ minus one. It is also clear that $\mathfrak S_n$ permutes the cells, so a cellular decomposition of $(X'/\mathfrak S_n, \pt)$ is induced. On the level of trees it corresponds to forgetting the labels on the leafs.

This cellular decomposition of $(X'/\mathfrak S_n, \pt)$ has only one cell $\sigma$ of minimal dimension $n+d-1$, that corresponds to the tree with $1$ vertex on each of levels $0,\ldots, d-1$, and $n$ vertices on the bottom level. The cells of dimension $n+d$ correspond to the trees with $1$ vertex on levels $0,\ldots, d-2$, $2$ vertices on level $d-1$, and $n$ vertices on level $d$. The coefficients of the boundary operator between $n+d$-dimensional and $n+d-1$-dimensional cells has the form $\pm\binom{n}{k}$ ($k=1,\ldots,n-1$), see~\cite[Theorem~2.5.1]{vass1988}; this is true for coefficients $Z_p$, or $\pm Z_p$, as we need.

If $n$ is a prime power $n=p^k$, then all the coefficients $\pm\binom{n}{k}$ are zero modulo $p$. Hence, the minimal cell gives a nontrivial element of homology $\sigma\in H_{n+d-1} (X'/{\mathfrak S_n}, \pt, \pm Z_p)$ (in fact it is also true for both coefficients $\pm Z_p$ and $Z_p$) and its Poincar\'e-Lefschetz dual $\xi$ is a nontrivial element of $H_{\mathfrak S_n}^{(d-1)(n-1)}(X, N)$. 

In fact $\xi$ is an image of the element of $H^*(B\mathfrak S_n, N)$ that is the $d-1$-th power of the Euler class of the standard $n-1$-dimensional irreducible representation $W$ of $\mathfrak S_n$. This can be shown by considering the map $\pi : X\to \mathbb R^{(d-1)n}$, that forgets the last coordinate of every point $x_i$, this map is $\mathfrak S_n$-equivariant, its target space is $W^{d-1}$ as $\mathfrak S_n$-representation, and its zero set $\sigma$ (note that all zeros are nondegenerate) must be dual to the Euler class $e(W)^{d-1}$.

Thus by definition $\uhind_N X = (d-1)(n-1)$.
\end{proof}

\section{Existence of symmetric configurations}
\label{symm-config}

In this section we consider certain equations for a system of functions on some configuration space, for such equations we prove existence of their solutions. 

Let us describe the idea more precisely. The results of this section are mostly inspired by the theorems of inscribing regular figures. One example of such theorems is the famous theorem of Schnirelmann~\cite{schn1934} that every simple smooth closed curve in $\mathbb R^2$ has an inscribed square. Some more results on inscribing $Z_p$-symmetric configurations are found in the paper~\cite{mak1989} and other papers of V.V.~Makeev.

Unlike the original problems on inscribing a congruent (or similar) copy of a given configuration, we consider here configurations of points on a manifold and try to find configurations with some metric equalities, which are not required to determine the configuration rigidly. The following particular result has explicit geometric meaning.

\begin{thm}
\label{metric-eq}
Let $p>2$ be a prime, let $M$ be an oriented closed smooth manifold of dimension $d$. Consider a continuous function $\rho:M\times M\to \mathbb R$.

Suppose that we have $d$ elements $g_1,\ldots, g_d$ of the group $G=Z_p$. Then there exists a nonconstant map $\phi: G\to M$ such that (group operation in $G$ is denoted $+$)
$$
\forall g\in G,\ \forall i=1,\ldots, d\quad \rho(\phi(g), \phi(g + g_i)) = \rho(\phi(e), \phi(g_i)).
$$
\end{thm}

In the case when $\rho$ is some continuous metric, we obtain more equalities because $\rho$ is symmetric ($\rho(x, y) = \rho(y, x)$). Moreover, in this case the numbers $\rho(\phi(g), \phi(g + g_i))$ have to be positive.

More specially, in the case ($d=2$, $G=Z_5$) Theorem~\ref{metric-eq} gives the following statement: for any continuous metric $\rho$ on a two-dimensional oriented manifold there are five distinct points $p_1,p_2,\ldots,p_5$ such that
$$
\rho(p_1p_2)=\rho(p_2p_3)=\rho(p_3p_4)=\rho(p_4p_5)=\rho(p_5p_1)
$$
and
$$
\rho(p_1p_3)=\rho(p_3p_5)=\rho(p_5p_2)=\rho(p_2p_4)=\rho(p_4p_1).
$$
In fact, Theorem~\ref{func-eq} (see below) tells, that in the two above equalities we may take two distinct metrics $\rho_1$ and $\rho_2$.

Let us state the results more generally. Theorem~\ref{metric-eq} has the following general form, where the functions $\rho$ are replaced by arbitrary continuous functions on the configuration space.

\begin{thm}
\label{func-eq}
Let $p>2$ be a prime, let $M$ be an oriented closed smooth manifold of dimension $d$, let $G=Z_p$, let $\alpha_i : \Map(G, M)\to \mathbb R$ ($i=1,\ldots,d$) be some continuous functions on the configuration space.

Then there exists a configuration $\phi\in \Map_\Delta(G, M)$ such that for any $i=1,\ldots, d$ the number $\alpha_i(\phi^g)$ does not depend on $g\in G$.
\end{thm}

Now we replace the special constraint system in Theorem~\ref{func-eq} by some arbitrary constraint system and formulate the following result.

\begin{thm}
\label{constr-func-eq}
Let $p>2$ be a prime, let $M$ be a smooth manifold of dimension $d$, let $G=(Z_p)^k$, and let $\mathcal S$ be some constraint system. Put
$$
m=\left\lfloor\frac{(d-1)(p^k-p^{k-1})+ w(\mathcal S) - 2}{p^k-1}\right\rfloor.
$$

Consider some $m$ continuous functions $\alpha_i : V(G, M, \mathcal S)\to \mathbb R$.

Then there exists a configuration $\phi\in V(G, M, \mathcal S)$ such that for every $i=1,\ldots, m$ the number $\alpha_i(\phi^g)$ does not depend on $g\in G$.
\end{thm}

In this theorem the number of functions decreased compared to Theorem~\ref{func-eq}, but the constraint system can be arbitrary and the manifold does not have to be closed or oriented.

\begin{proof}[Proof of Theorem~\ref{func-eq}]
For every $\alpha_i$ in Theorem~\ref{func-eq} consider the $G$-map of the configuration space $f_i: \Map_\Delta(G, M)\to \mathbb R^q$ by the formula
$$
f_i(\phi) = \oplus_{g\in G} \alpha_i(\phi^g).
$$
Consider the diagonal $\Delta(\mathbb R)\in\mathbb R^q$ and the space $\mathbb R^q/\mathbb R = W_i$, $f_i$ induces the map $h_i : \Map_\Delta(G, M) \to W_i$. 

We have to prove that the total map $h = h_1\oplus\dots\oplus h_d$ maps some point $\Map_\Delta(G, M)$ to zero in $W = W_1\oplus\dots\oplus W_d$. The latter space has $G$-action, so $h$ can be considered as $G$-equivariant section of a $G$-bundle over $\Map_\Delta(G, M)$.

The considered $G$-bundle over $\Map_\Delta(G, M)$ is a pullback of the $G$-bundle $W$ over $\pt$. The latter bundle (representation) has nonzero Euler class in $H_G^{d(p-1)}(\pt)$ (see~\cite{hsiang1975,vol1992}). Since $i(\Map_\Delta(G, M)) \ge d(p-1) + 1$, the image of this Euler class in $H_G^{d(p-1)} (\Map_\Delta(G, M))$ is nonzero too. This guarantees the existence of a zero and Theorem~\ref{func-eq} is proved.
\end{proof}

Theorem~\ref{constr-func-eq} is deduced from Lemma~\ref{constr-conf-ind} in the similar way.

\section{The genus and the category}
\label{genus-and-cat}

The estimates in the cohomological index of the configuration spaces in Section~\ref{index-conf} give estimates for the genus in the sense of Krasnosel'skii-Schwarz and the equivariant Lyusternik-Schnirelmann category of those spaces.

Let us start from the Lyusternik-Schnirelmann category. We formulate some special cases of definitions from the book~\cite{bart1993}.

\begin{defn}
Let $X$ be a $G$-space, \emph{$G$-category} of $X$ is the minimal size of $G$-invariant open cover (i.e. cover by $G$-invariant open subsets) $\{X_1,\ldots,X_n\}$ of $X$ such that every inclusion map $X_i\to X$ is $G$-homotopic to inclusion of some orbit $G/H\to X$. Denote $G$-category of $X$ by $\Gcat X$.
\end{defn}

This category gives a lower bound on the number of $G$-orbits of critical points for some $G$-invariant $C^2$-smooth function $f$ on $X$, if $f$ is a proper function bounded either from below or from above. One of the main ways to find lower bounds for $G$-category is to use $G$-genus, introduced in~\cite{kr1952,schw1957} for free $G$-actions, in~\cite{schw1966} for fiber bundles, and in~\cite{clp1986,clp1991} for arbitrary $G$-action. In fact there are different types of genus (see also the book~\cite{bart1993} for a detailed discussion), here we use one certain type.

\begin{defn}
Let $\mathcal A$ be some family of $G$-spaces. Let $X$ be a $G$-space, \emph{$\mathcal A$-genus} of $X$ is the minimal size of $G$-invariant open cover (i.e. cover by $G$-invariant open subsets) $\{X_1,\ldots,X_n\}$ of $X$ such that every $X_i$ can be $G$-mapped to some $D\in\mathcal A$. Denote $\mathcal A$-genus of $X$ by $\Agen X$.
\end{defn}

Equivalently (see~\cite{bart1993}), for paracompact spaces and finite groups $G$ the genus can be defined in the following way. Note that in this paper we consider paracompact spaces and finite groups only.

\begin{defn}
Let $X$ be a $G$-space, \emph{$\mathcal A$-genus} of $X$ is the minimal $n$ such that $X$ can be $G$-mapped to a join $D_1*D_2*\dots*D_n$, where $D_i\in\mathcal A$.
\end{defn}

The following theorem estimates the equivariant category by the genus, it follows directly from the definitions.

\begin{thm}
\label{cat-by-genus}
For any $G$-space $X$ denote $\mathcal O_G(X)$ the set of distinct types of orbits $G/H\subseteq X$. Then
$$
\Gcat X\ge \OGXgen X.
$$
\end{thm}

The genus $\OGXgen X$ is usually estimated by the following type of genus.

\begin{defn}
If the family $\mathcal A$ contains only one $G$-space, which is the disjoint union of all nontrivial orbit types 
$$
D_G = \bigsqcup_{H\subset G} G/H,
$$
then we denote $\Agen X = \Ggen X$.
\end{defn}

In the paper~\cite{vol2007} this genus is denoted $g_G(X)$. In the sequel we shall mainly use this genus by the following reason. If a $G$-space has $G$-fixed points, then $\OGXgen X=1$. If a $G$-space has no fixed points, then $\OGXgen X\ge \Ggen X$. Still, for free $G$-spaces we use the $\OGXgen X$ itself. 

Now let us state some estimates on the genus of certain configuration spaces.

\begin{thm}
\label{nonconst-conf-gen}
Suppose that $G=(Z_p)^k$, $M$ is a smooth oriented (if $p>2$) closed manifold of dimension $d$. Then
$$
\Ggen \Map_\Delta(G, M) \ge d(p^k-1) + 1.
$$
\end{thm}

\begin{thm}
\label{constr-conf-gen}
Suppose that $G=(Z_p)^k$, $\mathcal S$ is a constraint system on $G$. Then for any $d$-dimensional smooth manifold $M$
$$
\Ggen V(G, M, \mathcal S) \ge (d-1)(p-1)p^{k-1} + w(\mathcal S) - 1.
$$
\end{thm}

In order to prove the estimates on the genus, we need some well-known lemmas. In the case of a free $G$-action of arbitrary finite group $G$ the genus can be estimated in the following way (see~\cite{schw1957}).

\begin{lem}
\label{gen-by-index}
If $G$ acts freely on $X$ then
$$
\OGXgen X\ge \uhind_G X + 1.
$$
\end{lem}

It should be mentioned that the above lemma is a simple consequence of the second definition of genus and the fact that $H^*_G(X, M) = H^*(X/G, M)$ for a free $G$-space $X$. The following lemma is clear from the definition of genus and existence of an $H$-map $G\to H$, where $H$ acts on $G$ by left multiplications.

\begin{lem}
\label{gen-group-change}
If $G$ acts freely on $X$ and $H\subset G$ is a subgroup, then
$$
\OGXgen X\ge \OHXgen X.
$$
\end{lem}

There is a good estimation of the genus by $i_G(X)$, we restate Proposition~4.7 from the paper~\cite{vol2005}, noting Remark~4.5 from the same paper.

\begin{lem}
\label{gen-by-i}
Suppose $G=(Z_p)^k$ acts on $X$ without fixed points. Then 
$$
\Ggen X\ge i_G(X).
$$
\end{lem}

Now Theorem~\ref{nonconst-conf-gen} is deduced from Lemmas~\ref{gen-by-i} and \ref{conf-space-ind}, and Theorem~\ref{constr-conf-gen} is deduced from Lemmas~\ref{gen-by-i} and \ref{constr-conf-ind}.

Let us state some corollary of Theorem~\ref{constr-conf-gen}. 

\begin{cor}
\label{pwconstr-gen0}
Let $G=(Z_p)^k$, let $M$ be a smooth manifold of dimension $d$, let the constraint system $\mathcal S$ consist of all pairs $\mathcal S = \binom{G}{2}$.  Then
$$
\Ggen V(G, M, \mathcal S) \ge (d-1)(p-1)p^{k-1}+1.
$$
\end{cor}

\begin{defn}
Denote the group of permutations of $[n]$ by $\mathfrak S_n$.
\end{defn}

In the papers~\cite{fuks1970,vass1988,rot2008} the classical configuration space $V\left(n, M, \binom{[n]}{2}\right)$ was considered, and the genus of $V\left(n, M, \binom{[n]}{2}\right)$ with respect to $\mathfrak S_n$-action was estimated from below. The estimates for $n$ being a prime power in the case $d=2$ and for $n=2^k$ in the case of arbitrary $d$ were better than in Corollary~\ref{pwconstr-gen0}. Here we prove the appropriate result for $\mathfrak S_n$-genus.

\begin{defn}
For a positive integer $n$ and a prime $p$ denote $D_p(n)$ the sum of digits in the $p$-ary representation of $n$.
\end{defn}

\begin{thm}
\label{pwconstr-gen}
Let $M$ be a smooth manifold of dimension $d\ge 2$, let the constraint system $\mathcal S\subset 2^n$ consist of all pairs $\mathcal S = \binom{[n]}{2}$. Denote $X=V(n, M, \mathcal S)$. Then for every prime $p$ we have
$$
\OSnXgen X \ge (d-1)(n-D_p(n)) + 1.
$$
\end{thm}

Here we formulate the result for arbitrary manifold, but actually we give prove for the case $M=\mathbb R^d$, as in the previous works. The author does not know whether the genus
$$
\OSnXgen V(n, M, \mathcal S)
$$
can be replaced by 
$$
\Zpkgen V(n, M, \mathcal S)
$$ 
in this theorem for the case when the number of points is a prime power $n=p^k$, as it is done in the statement of Corollary~\ref{pwconstr-gen0}.

\begin{proof}[Proof of Theorem~\ref{pwconstr-gen}]
The theorem for prime powers is deduced from Lemmas~\ref{gen-by-index} and \ref{pwconstr-ind} directly. 

Consider the general case: take some $n$ and represent it as a sum of $D_p(n)$ powers of $p$:
$$
n = \sum_{i=1}^{D_p(n)} p^{k_i}.
$$

As in~\cite{vass1988}, we can select $D_p(n)$ disjoint open subsets $M_1,\ldots, M_{D_p(n)}$ of $M$, each being homeomorphic to $\mathbb R^d$. Consider the subspace of $V(n, M, \mathcal S)$ formed by the following configurations: the first $p^{k_1}$ points lie in $M_1$, the other $p^{k_2}$ points lie in $M_2$ and so on. So we can state that 
$$
V(n, M, \mathcal S) \supset \prod_{i=1}^{D_p(n)} V(p^{k_i}, M_i, \mathcal S).
$$
The spaces $X_i=V(p^{k_i}, M_i, \mathcal S)$ have actions of their respective groups $G_i=\mathfrak S_{p^{k_i}}$, and have by Lemma~\ref{pwconstr-ind} some nontrivial cohomology classes in $H_{G_i}^{(d-1)(p^{k_i}-1)}(X_i, N)$ as natural images of some classes in $H^*(BG_i, N)$, $N$ is the same as in Lemma~\ref{pwconstr-ind}. By the K\"unneth formula the product of these classes is a nontrivial natural image of some class from $H^{(d-1)(n-D_p(n))}(BG_1\times\dots\times BG_{D_p(n)}, N)$. Denote $G=G_1\times\dots\times G_{D_p(n)}$, so we see that for the upper $G$-index
$$
\uhind_G V(n, M, \mathcal S)\ge\uhind_G \prod_i X_i\ge (d-1)(n-D_p(n)).
$$
From Lemma~\ref{gen-by-index} it follows that if $X=V(n,M,S)$, then $\OGXgen X\ge (d-1)(n-D_p(n))+1$. Then by Lemma~\ref{gen-group-change} $\OSnXgen X\ge \OGXgen X\ge (d-1)(n-D_p(n))+1$.
\end{proof}

\end{document}